\documentclass[12pt]{amsart}
\usepackage{amssymb}
\newcommand{\F}{\mathbb{F}}
\newcommand{\C}{\mathbb{C}}
\newcommand{\GL}{\operatorname{GL}}

\newcommand{\Z}{\mathbb{Z}}
\newcommand{\g}{\mathfrak{g}}
\newcommand{\gl}{\mathfrak{gl}}
\renewcommand{\sl}{\mathfrak{sl}}
\newcommand{\X}{\mathfrak{X}}
\newcommand{\Rep}{\operatorname{Rep}}
\newcommand{\Cat}{\mathcal{C}}
\newcommand{\Fi}{\mathcal{F}}
\newcommand{\Hom}{\operatorname{Hom}}
\newcommand{\End}{\operatorname{End}}
\newtheorem{Thm}{Theorem}[section]
\newtheorem{Prop}[Thm]{Proposition}

\newtheorem{Lem}[Thm]{Lemma}
\theoremstyle{definition}

\address{Department
of Mathematics, Northeastern University, Boston MA 02115 USA}
\email{i.loseu@neu.edu}
\thanks{Supported by the NSF grant DMS-0900907}
\title{Representations of  general linear groups and categorical actions
of Kac-Moody algebras}
\author{Ivan Losev}
\begin{document}
\begin{abstract}
This is an expanded version of the lectures given by the author on the 3rd  school
``Lie algebras, algebraic groups and invariant theory'' in Togliatti, Russia.
In these notes we explain the concept of a categorical Kac-Moody action by studying an
example of the category of rational representations of a general linear group
in positive characteristic. We also deal with some more advanced topics: a categorical
action on the polynomial representations and crystals of categorical actions.
\end{abstract}
\maketitle
\begin{center}{\it Dedicated to Ernest Borisovich Vinberg on  his 75th birthday.}\end{center}
\markright{REPRESENTATIONS AND CATEGORICAL ACTIONS}
\tableofcontents
\section{Introduction}
Categorical actions of Kac-Moody algebras is a relatively young subject that arises in Representation
theory and in Knot theory. The first formal definition appeared in a paper of Chuang and Rouquier,
\cite{CR} in the case of $\mathfrak{sl}_2$. The general case was treated in \cite{Rouquier_2Kac}
and also in the work of Khovanov and Lauda, \cite{KL1}-\cite{KL3}.

The ideas leading to categorical actions were around for some two decades, they appeared in
the work of Kleshchev, \cite{Kleshchev_br1,Kleshchev_br2}, Ariki, \cite{Ariki},
Lascoux-Leclerc-Thibon, \cite{LLT}, Okounkov-Vershik, \cite{OV}, Bernstein-Frenkel-Khovanov,
\cite{BFK}, Brundan-Kleshchev, \cite{BK_restr},\cite{BK_gen_lin}, and  others.
In the aforementioned papers it was observed that many categories occurring  in Representation
theory, such as the representations of symmetric groups, of Hecke algebras, of the general
linear groups or of Lie algebras of type A  have endo-functors that on the level of the Grothendieck group give
actions of Kac-Moody Lie algebras of type $A$. Moreover, the endofunctors come
equipped with some natural transformations. These ideas have lead to the definition of
categorical $\mathfrak{sl}_2$-actions that can be easily generalized to arbitrary
type A algebras ($\mathfrak{sl}_n, \hat{\mathfrak{sl}}_n$ or $\mathfrak{gl}_{\infty}$).

In these notes we provide an introduction to categorical Kac-Moody actions by considering
an example: a categorical action of the affine Kac-Moody algebra $\hat{\sl}_p$ on the category $\Rep(G)$
of rational representations of the general linear group $G:=\operatorname{GL}_n$ over an
algebraically closed field of characteristic $p$. After fixing some notation in
Section \ref{S_not}, in Section \ref{S_compar} we compare the representation theories of $G$ in characteristic
$0$ and in characteristic $p$. The characteristic $0$ the story is classical and easy:
all representations are completely reducible and the irreducibles are classified by dominant weights.
The characteristic $p$ story is much more complicated. We still have analogs of
irreducible modules in characteristic $0$ (the Weyl modules) but they are no
longer reducible.

The next two sections are devoted to a categorical $\hat{\sl}_p$-action on
$\Rep(G)$ essentially introduced in \cite{BK_gen_lin}. In Section \ref{S_cat_fun}
we introduce exact endo-functors $F_\alpha, E_\alpha$  of $\Rep(G)$, where $\alpha$ runs
over the filed of residues mod $p$. These functors
are direct summands of the functors of tensor products with the tautological
$G$-module (for the $F$'s) and with its dual (for the $E$'s). On the Grothendieck
group, these functors define a representation of $\hat{\sl}_p$.

The functors are not the only data required to define a categorical action.
In addition, one needs certain functor morphisms. These are discussed in
Section \ref{S_cat_act}, where a definition of a categorical action
(in the special case of $\hat{\sl}_p$) due to Rouquier, \cite{Rouquier_2Kac},
is given.

In Section \ref{S_cryst} we describe a natural crystal associated to a categorical
action and explicitly describe the crystal associated to $\Rep(G)$.

In Section \ref{S_pol} we discuss two more categorical actions: a more standard, on
the representations of the symmetric groups, and a less standard, on polynomial representations
of the general linear groups. We also show that the Schur functor becomes a morphism of categorical
actions. 

Finally, in a very short Section \ref{S_next} we briefly discuss some further developments.

{\bf Acknowledgements}. I'd like to thank A. Kleshchev for numerous discussions related to 
various topics of these lectures and O. Yacobi for his remarks on a preliminary version of this text. 
Also I want to thank I. Arzhantsev for inviting to lecture in the summer school ``Lie algebras, algebraic groups and invariant theory''
in Togliatti, and R. Uteeva for her care of the conference participants. 

\section{Notation}\label{S_not}
Let $\F$ be an algebraically closed field. Let $n$ be a positive integer and $V$ be an $n$-dimensional
$\F$-vector space. We consider the general linear group $G:=\GL(V)$ and its Lie algebra $\g=\gl(V)$.
To $G$ we can assign the category $\operatorname{Rep}(G)$ of its rational representations, i.e., of all
finite dimensional representations whose matrix elements are regular functions on the algebraic group $G$
(more precisely, polynomials in the matrix coefficients on $G$ and $\det^{-1}$).

Inside $G$, we consider the subgroups $T$ of all diagonal matrices (a maximal torus) and $B$
of all upper-triangular matrices (a Borel subgroup). The character groups $\operatorname{Hom}(B,\F^\times),
\operatorname{Hom}(T,\F^\times)$ are naturally identified, we denote this group by $\mathfrak{X}$.
We can identify $\X$ with $\Z^n$: to a character $\chi$ we assign an $n$-tuple $(\chi_1,\ldots,\chi_n)$
such that $\chi$ maps $t=\operatorname{diag}(t_1,\ldots,t_n)$ to $t_1^{\chi_1}t_2^{\chi_2}\ldots t_n^{\chi_n}$.
Inside $\X$ one can consider the subset $\X^{+}$ of {\it dominant weights}, $\X^{+}=\{\lambda\in \X| \lambda_1\geqslant \lambda_2\geqslant \ldots\geqslant \lambda_n\}$.

We can equip $\X$ with a partial order: we say that $\chi\leqslant \chi'$ if $\sum_{i=1}^k \chi_i\leqslant \sum_{i=1}^k \chi_i'$
for all $k=1,2,\ldots, n-1$ and $\sum_{i=1}^n \chi_i=\sum_{i=1}^n \chi_i'$.

Starting from Subsection \ref{SS_functors} we assume, for simplicity, that the characteristic of $\F$ is different from $2$.

\section{Characteristic 0 vs characteristic $p$}\label{S_compar}
In this section we will compare the representation theory of $G$ in the case when $\F$ has characteristic $0$
(a well-known case) and in the case when the characteristic is positive (less known).
\subsection{Representations of a torus}
Regardless the characteristic, any rational representation $M$ of $T$ decomposes into the sum of its {\it weight spaces}:
$M=\bigoplus_{\chi\in \X}M_\chi$, where, by definition, $M_\chi:=\{m\in M| t.m=\chi(t)m, \forall t\in T\}$,
where $t.m$ stands for the image of $m$ under the action of $t$.

\subsection{Groups vs Lie algebras}
Now let $M\in \Rep(G)$. Then on $M$ we have a natural representation of the Lie algebra $\g$.
Therefore $M$  becomes a module over the universal enveloping algebra $U(\g)$.

When the characteristic is $0$, the representation theories of $G$ and $U(\g)$ are very tightly related.
Namely, two representations of $G$ giving the same structure of a $U(\g)$-module are isomorphic.
If instead of $\GL(V)$ we take $\operatorname{SL}(V)$ (or any other semisimple simply connected group),
then any finite dimensional representation of $\g$ comes from a representation of $G$. For $G=\GL_n(V)$
both $G$ and $\g$ have one-dimensional centers and there is more freedom in defining a representation of
the center in the Lie algebra setting. A representation $N$ of $\g$ comes from a representation of $G$
if and only if the Lie algebra $\mathfrak{t}$ of $T$ acts on $N$ diagonalizably with integral eigenvalues
(i.e., all matrix units $E_{ii}$ are diagonalizable operators with integral eigenvalues)\footnote{This description
is not optimal but it will do for our purposes.}.

In characteristic $p$, a connection between the representations of $G$ and of $\g$ is much more loose.
There are lots of representations of $\g$ (even for $G=\operatorname{SL}(V)$) that do not come
from representations of $G$. On the other hand, non-isomorphic representations of $G$ can produce
the same representation of $\g$, this is because of the Frobenius automorphism.

However, if one replaces $U(\g)$ with a somewhat different algebra, one still recovers familiar
results from characteristic $0$. That algebra is called a {\it hyperalgebra} and is constructed
as follows. Consider the universal enveloping algebra $U(\g(\mathbb{Q}))$. Inside consider the {\it divided
power} subring $\dot{U}(\g(\Z))$ generated by the divided powers $E_{ij}^{(n)}:=\frac{E_{ij}^n}{n!}, i\neq j,$
and the binomial coefficients $\binom{E_{ii}}{n}:=\frac{E_{ii}(E_{ii}-1)\ldots (E_{ii}-n+1)}{n!}$.
Then set $\dot{U}(\g):=\F\otimes_{\Z} \dot{U}(\g(\Z))$. The structure of this algebra is very different
from the usual universal enveloping algebra. For example, the algebra $\dot{U}(\g)$ is not finitely generated,
it contains nilpotent elements (e.g., $E_{ij}$ with $i\neq j$ satisfy $E_{ij}^p=p! E_{ij}^{(p)}=0$)
and even any its finitely generated subalgebra is finite dimensional.

If $G=\operatorname{SL}(V)$, then a rational representation of $G$ is the same as a finite dimensional
$\dot{U}(\g)$-module. For $G=\GL(V)$ we need to impose a certain integrality condition analogous
to the above. This can be done as follows. To $\chi\in \X$ we can assign a character
of the hyperalgebra  $\dot{U}(\mathfrak{t})\subset \dot{U}(\g)$ (that is defined analogously
to $\dot{U}(\g)$) by the formula $\chi\binom{E_{ii}}{m}=\binom{\chi_i}{m}$, where the right hand side
is viewed as an element of $\F$. Then a $\dot{U}(\g)$-module $N$ comes from a representation of
$G$ if and only if $N=\bigoplus_{\chi\in \X} N_\chi$, where $N_\chi$ is the eigenspace for $\dot{U}(\mathfrak{t})$
corresponding to the character $\chi$.

Let us give a hint on how to produce a representation $\varphi: \dot{U}(\g)\rightarrow \operatorname{End}(V)$
from a rational representation $\Phi:G\rightarrow \GL(V)$. Observe that $\Phi(E+t E_{ij})$ is a polynomial in $t$. Then define $\varphi(E_{ij}^{(n)})$ as the coefficient of $t^n$ in $\Phi(E+tE_{ij})$.
This is supposed to replace the formula $\Phi(E+t E_{ij})=\exp(\varphi(t))$ that no longer makes sense.
To recover $\varphi(\binom{E_{ii}}{n})$ one looks at $\Phi(\operatorname{diag}(1,\ldots 1,t,1\ldots,1)), t\neq 0$.

To  finish this discussion, let us remark that $\dot{U}(\g)$ makes prefect sense in the characteristic $0$
case (with the same definition) but, of course, there $\dot{U}(\g)=U(\g)$.

\subsection{Weyl modules}
We will use a connection between the representations of $G$ and of $\dot{U}(\g)$ explained in the previous subsection
to produce certain representations of $G$, called the {\it Weyl modules}.

Fix $\lambda\in \X^+$. Consider the $\dot{U}(\g)$-module $\Delta(\lambda)$ generated by a single generator $v_\lambda$
(a.k.a. highest vector) and the following relations:
$$ E_{ii+1}^{(m)}v_\lambda=0,\,\, \binom{E_{ii}}{m}v_\lambda=\binom{\lambda_i}{m}v_\lambda, \forall m>0,\quad E_{i+1 i}^{(m)}v_\lambda=0, \forall m\geqslant \lambda_{i+1}-\lambda_i+1.$$

Clearly, the $\dot{U}(\g)$-module $\Delta(\lambda)$ satisfies the additional integrality condition above and hence gives
a representation of $G$.

To
$M\in \Rep(G)$ one assigns a formal character: $\operatorname{ch}M:=\bigoplus_{\chi\in \X} \dim M_\chi \cdot e^\chi$.
One can compute the character of $\Delta(\lambda)$, this is the standard Weyl character formula.
We have $$\operatorname{ch}\Delta(\lambda)=\frac{\sum_{w\in W} (-1)^{l(w)}e^{w(\lambda+\rho)}}{\sum_{w\in W} (-1)^{l(w)}e^{w\rho}}.$$
Here $W\cong S_n$ is the Weyl group of $G$, $\rho:=(n,n-1,\ldots,1)\in \X$ and, for $w\in W$, we write $l(w)$
for the length of $w$.

\subsection{Irreducible modules}
In characteristic $0$, the $G$-module $\Delta(\lambda)$ is irreducible. In  characteristic $p>0$, this is no longer so.
For example, take $\lambda=(p,0,\ldots,0)$. Consider the $G$-module $\F[x_1,\ldots,x_n]$ and its homogeneous component
$\F[x_1,\ldots,x_n]_p$ of degree $p$. One can show that $\F[x_1,\ldots,x_n]_p^*$  is a $G$-module isomorphic to $\Delta(\lambda)$.
Inside $\F[x_1,\ldots,x_n]_p$ we have a submodule $L$ spanned by $x_1^p,\ldots,x_n^p$. This submodule is clearly irreducible.
So we get a proper quotient of $\Delta(\lambda)$.

In any case, since $\Delta(\lambda)_\lambda$ is a one-dimensional vector space, the sum $R$ of all submodules of $\Delta(\lambda)$
not containing $v_\lambda$ is a proper submodule. The quotient $L(\lambda):=\Delta(\lambda)/R$ is an irreducible module.
One can easily show that the assignment $\lambda\mapsto L(\lambda)$ defines a bijection between $\X^+$
and the set $\operatorname{Irrep}(G)$ of irreducible representations of $G$. The inverse bijection sends
an irreducible module $L$ to the largest (with respect to the partial order introduced above) weight
$\lambda\in \X^+$ with $L_\lambda\neq \{0\}$.

In fact, the characters of $L(\lambda)$ are not known, in general, even for $G=\operatorname{GL}_n(\F)$.
To determine them is one of the most important problems in the modular representation theory\footnote{
There are conjectures of Lusztig on what happens for $p>n$ and they are proved for $p$ large enough
comparing to $n$. The proof is in three highly non-trivial parts that are important
in their own right: to establish character formulas in the full category
$O$ for affine Lie algebras
(Kashiwara-Tanisaki), to relate certain parabolic categories $O$ for affine Lie algebras
to the categories of finite dimensional representations of quantum groups (Kazhdan-Lusztig),
and then to pass from quantum groups (in characteristic 0) to algebraic groups in characteristic
$p$ (Andersen-Jantzen-Soergel). This stuff is far beyond the scope of these lectures.}.

\subsection{Complete reducibility}\label{SS_compl_red}
In characteristic $0$, the category $\operatorname{Rep}(G)$ is semisimple, i.e., any representation is
completely reducible. In  positive characteristic, this is no longer true, in fact, the module
$\Delta(\lambda)$ for $\lambda=(p,0,\ldots,0)$ is not completely reducible: one cannot split
the projection $\Delta(\lambda)\twoheadrightarrow L(\lambda)$ because $\Delta(\lambda)$ is generated
by $v_\lambda$.

In fact, one can still characterize $\Delta(\lambda)$ by a universal property: for any $M\in \Rep(G)$
and any vector $n\in N$ such that $b.n=\lambda(b)n$ for all $b\in B$, one has a unique homomorphism
$\Delta(\lambda)\rightarrow N$ with $v_\lambda\mapsto n$. In particular, the endomorphism space of $\Delta(\lambda)$
is $\F$.

We remark that these properties of $\Delta(\lambda)$ occur also in a more classical situation: for Verma modules over a complex
semisimple Lie algebra. In fact, there are more common features: both $\Rep(G)$ and the BGG category
$\mathcal{O}$ are {\it highest weight categories}. Very informally, this means that a half of the complete
reducibility survives (so ``highest weight''=``quarter-simple''). In particular,  the multiplicity
$[\Delta(\lambda):L(\mu)]$ of $L(\mu)$ in (the composition series of) $\Delta(\lambda)$ is zero unless
$\mu\leqslant \lambda$, moreover, $[\Delta(\lambda):L(\lambda)]=1$. This follows from considering the
weights. Also we note that any exact sequence $0\rightarrow \Delta(\lambda)\rightarrow M\rightarrow \Delta(\mu)\rightarrow 0$
splits unless $\lambda>\mu$. This is a consequence of the universality property of $\Delta(\mu)$.

In addition, the axioms of a highest weight category require the existence of enough projectives that
have to be filtered, with subsequent quotients being {\it standard} (i.e. Weyl/Verma) modules. More precisely,
for any $\lambda\in \X^+$ there is a projective cover $P(\lambda)$ of $\Delta(\lambda)$ such that the kernel
of $P(\lambda)\twoheadrightarrow \Delta(\lambda)$ admits a filtration with subquotients of the form
$\Delta(\mu)$ with $\mu>\lambda$. Mostly, we will not need this.

\section{Categorification functors}\label{S_cat_fun}
When we speak about Lie algebra actions on vector spaces we mean that algebras act by linear operators.
We want an action of the Kac-Moody algebra $\hat{\sl}_p$. This algebra is generated by elements $e_\alpha,f_\alpha$, where $\alpha $
ranges over the simple field $\F_p\subset \F$ subject to certain relations. So to define a representation
of $\hat{\sl}_p$ on a vector space $W$ we need to equip $W$ with operators $e_\alpha^W, f_\alpha^W$
that satisfy the relations.

On the categorical level, we should have an action on a category by functors. More precisely, for each $\alpha\in \F_p$,
we need functors $E_\alpha, F_\alpha$ that ``categorify'' $e_\alpha, f_\alpha$. To make sense of the
word in quotation marks let us recall that from an abelian category $\Cat$  we can construct a complex vector space,
a complexified Grothendieck group $[\Cat]$. An exact endofunctor of $\Cat$ produces a linear operator on $[\Cat]$.
So, provided the functors $E_\alpha,F_\alpha$ are exact, we get operators $[E_\alpha],[F_\alpha]$. The first thing
that we mean when we say that $E_\alpha, F_\alpha$ produce a categorical $\hat{\sl}_p$-action
is that $[E_\alpha],[F_\alpha]$ define a representation of $\hat{\sl}_p$ in the usual sense.

Let us explain how to produce functors $E_\alpha, F_\alpha$. They will be constructed as direct summands
of functors $F:=V\otimes \bullet, E:=V^*\otimes \bullet$. More precisely, they will arise as eigen-functors
for a natural transformations of $E,F$ coming from a tensor Casimir.

\subsection{Tensor products with $V$ and $V^*$}
Recall that $G$ stands for $\GL(V)$, where $V=\F^n$.  For $M\in \Rep(G)$ we set
$F(M)=V\otimes M, E(M)=V^*\otimes M$. Then we can view $E,F$ as functors $\Rep(G)\rightarrow \Rep(G)$.
These functors are clearly exact. Furthermore they are biadjoint: we have functorial isomorphisms
$\Hom(V\otimes M, N)\cong \Hom(M,V^*\otimes N), \Hom(V^*\otimes M, N)=\Hom(M,V\otimes N)$.

In characteristic $0$, one can describe the structure of $F(M), E(M)$ completely, thanks to
the semi-simplicity. It is enough to compute these representations for $M=\Delta(\lambda)$.
To state the result we need some notation. Namely, for $\lambda\in \X^+$, let $I^+_\lambda$
be the set of indices $i\in \{1,2,\ldots,n\}$ such that $\lambda+\epsilon_i\in \X^+$,
where $\epsilon_i$ is the coordinate vector at $i$. In other words, $I^+_{\lambda}=\{1\}\sqcup\{i\in \{2,\ldots,n\}, \lambda_{i-1}>\lambda_i\}$. Similarly,  we can consider the subset $I^-_\lambda:=\{i\in \{1,2,\ldots,n\}| \lambda-\epsilon_i\in \X^+\}$.

Then we have
$$F(\Delta(\lambda))\cong\bigoplus_{i\in I^+_\lambda} \Delta(\lambda+\epsilon_i),\quad
E(\Delta(\lambda))\cong\bigoplus_{i\in I^-_\lambda} \Delta(\lambda-\epsilon_i).$$
To see this one can compute the characters of both sides and check that they are equal.

Again, in characteristic $p$, the situation is more complicated: instead of the decomposition into
a direct sum, we have a filtration. Namely, let us write the elements of $I^+_\lambda$ in the increasing order:
$1=i_1<i_2<\ldots<i_k$. Then there is a filtration $0\subset \Fi_{i_1}\subset \Fi_{i_2}\subset\ldots \subset \Fi_{i_k}=F(\Delta(\lambda))$
by $G$-submodules such that $\Fi_{i_j}/\Fi_{i_{j-1}}=\Delta(\lambda+\epsilon_{i_j})$ for all $j$.

Let us prove this. Let $v_\lambda$ be a highest vector in $\Delta(\lambda)$ and let $v_1,\ldots,v_n$ be the tautological
basis of $V=\F^n$. Of course, $v_1\otimes v_\lambda\in F(\Delta(\lambda))$ is a highest vector of weight $\lambda+\epsilon_1$.
By the universality property of $\Delta(\lambda+\epsilon_1)$, see Subsection \ref{SS_compl_red}, there is a homomorphism $\Delta(\lambda+\epsilon_1)\rightarrow F(\Delta(\lambda))$ mapping $v_{\lambda+\epsilon_1}$ to $v_1\otimes v_\lambda$.
If $2\not\in I^+_\lambda$, then $E_{21}v_\lambda=0$ and so $E_{21}(v_1\otimes v_\lambda)=v_2\otimes v_\lambda$. Therefore
$v_2\otimes v_\lambda\in \operatorname{im}\Delta(\lambda+\epsilon_1)$. Arguing in this way, we see that
$v_j\otimes v_\lambda\in \operatorname{im}\Delta(\lambda+\epsilon_1)$ for $j<i_2$. We also see that
$v_{i_2}\otimes v_\lambda$ is a highest vector of weight $\lambda+\epsilon_{i_2}$ modulo
the image of $\Delta(\lambda+\epsilon_1)$. Continuing this argument, we see that we almost have a
filtration as needed: with subsequent quotients being homomorphic images of $\Delta(\lambda+\epsilon_i), i\in I^+_\lambda$.
But now the character computation shows that the homomorphic images should be $\Delta(\lambda+\epsilon_i)$ themselves.

Similarly, let $n=i_1>i_2>\ldots>i_k$ be the elements of $I^-_\lambda$. Then there is a filtration $0\subset \Fi_{i_1}
\subset \Fi_{i_2}\subset\ldots\subset \Fi_{i_k}=E(\Delta(\lambda))$ such that $\Fi_{i_j}/\Fi_{i_{j-1}}=\Delta(\lambda-\epsilon_{i_j})$.

\subsection{Tensor Casimir}
We need to decompose the functor $F$ into a direct sum of functors $F_\alpha, \alpha\in \F_p$. In other words, we need
to produce a decomposition $F(M)=\bigoplus_{\alpha\in \F_p}F_\alpha(M)$ into the sum of $G$-modules that is functorial in $M$.
One way to get such a decomposition is to pick some linear operator $X_M$ on $F(M)$ and for $F_\alpha(M)$
take the generalized eigen-space corresponding to the eigenvalue $\alpha$. Then $F_\alpha(M)$ will be
$G$-stable provided $X_M$ is $G$-equivariant. As for functoriality, let us notice that the tensor product
$U(\g)\otimes U(\g)$ acts on $F(M)=V\otimes M$ and for a $G$-module homomorphism $M\rightarrow N$ the corresponding
homomorphism $F(M)\rightarrow F(N)$ is $U(\g)\otimes U(\g)$-equivariant. So the eigen-decomposition for $X_M$
coming from some $X\in U(\g)\otimes U(\g)$ will be functorial.

The discussion above suggests that we want to pick some element $X\in [U(\g)\otimes U(\g)]^G$ and for $F_\alpha(M)$
take the generalized $\alpha$-eigenspace for the operator $X_M$ induced by $X$. It is not reasonable to take $X$
of the form $1\otimes \bullet$ or $\bullet\otimes 1$. So the simplest choice we can make is the {\it tensor Casimir}:
$$X:=\sum_{i,j=1}^n E_{ij}\otimes E_{ji},$$
this is an element of $[\g\otimes\g]^G$.

\subsection{Functors $F_\alpha$ and $E_{\alpha}$}\label{SS_functors}
Of course, we still need to check that the eigenvalues of $X_M$ on $F(M)$ (and on $E(M)$) are in $\F_p$. A key step here is the following proposition that also will be used below. For simplicity, from now on  we assume that $p>2$.

\begin{Prop}\label{Prop:casimir_filtr}
The operator $X_{\Delta(\lambda)}$ preserves the filtration $0\subset \mathcal{F}_{i_1}\subset \mathcal{F}_{i_2}\subset\ldots
\subset \Fi_{i_k}$ and acts on the quotient $\Delta(\lambda+\epsilon_i)$ by $\lambda_i+1-i$ (viewed as an element of $\F_p$).
\end{Prop}
\begin{proof}
The first claim follows from $\Hom(\mathcal{F}_i, F\Delta(\lambda)/\Fi_i)=0$ (where ``$\Hom$''
means the Hom space in the category $\Rep(G)$). Indeed, for any filtration subquotients
$\Delta(\mu)$ of $\Fi_i$ and $\Delta(\mu')$ of $F\Delta(\lambda)/\Fi_i$ we have $\mu>\mu'$. Therefore $\Hom(\Delta(\mu),\Delta(\mu'))=0$
and hence $\Hom(\mathcal{F}_i, F\Delta(\lambda)/\Fi_i)=0$.

The claim that $X_\Delta(\lambda)$ acts on $\Delta(\lambda+\epsilon_i)$ by a scalar simply follows from
$\operatorname{End}(\Delta(\lambda+\epsilon_i))=\F$. It remains to compute the scalar. For this recall
the usual Casimir $C=\sum_{i,j=1}^n E_{ij}E_{ji}\in U(\g)$. Then it is easy to see that  $$X=\frac{1}{2}(\delta(C)-1\otimes C-C\otimes 1),$$
where $\delta$ stands for the coproduct $U(\g)\rightarrow U(\g)\otimes U(\g)$ (this is precisely where we need the assumption
$p>2$). The element $C$ acts on $\Delta(\mu)$ by a scalar, and to compute it we notice that
$$C=2\sum_{i>j}E_{ij}E_{ji}+\sum_{i=1}^n E_{ii}(E_{ii}+n+1-2i).$$
The first summand acts on $v_\mu\in \Delta(\mu)$ by zero, while the second one multiplies  $v_\mu$ by
\begin{equation}\label{eq:scalar}\sum_{i=1}^n \mu_i(\mu_i+n+1-2i).\end{equation} So, on
$\Delta(\lambda+\epsilon_i)$, the operator $X_M$ acts by (notice that $V=\Delta(\epsilon_1)$)
$$\frac{1}{2}(C_{\Delta(\lambda+\epsilon_i)}-C_{\Delta(\epsilon_1)}-C_{\Delta(\lambda)}).$$
Then we can plug $\mu=\lambda+\epsilon_1,\lambda,\epsilon_1$ (\ref{eq:scalar}) and, simplifying the corresponding expression, get the scalar
$\lambda_i+1-i$.
\end{proof}

The proof that the eigenvalues of $X_M$ are in $\F_p$ now can be done as follows. Let $0\rightarrow M_1\rightarrow M_2\rightarrow M_3\rightarrow 0$ be an exact sequence of $G$-modules. Then $X_{M_1}=X_{M_2}|_{M_1}$ and $X_{M_3}$ is induced by $X_{M_2}$.
So the set of eigenvalues of $X_{M_2}$ is the union of such sets for $X_{M_1},X_{M_3}$. From the surjection $\Delta(\lambda)\twoheadrightarrow
L(\lambda)$ we deduce that the eigenvalues of $X_{L(\lambda)}$ are in $\F_p$. An arbitrary $M\in \Rep(G)$ has a Jordan-H\"{o}lder
series (and the subsequent quotients are $L(\bullet)$'s). The desired result for $X_M$ follows.

Now, by definition, $F_\alpha(M)$ is the generalized $\alpha$-eigenspace of $X_M$.

Below we will also need an alternative description of $F_\alpha(M)$. For $\beta\in \F_p$ and $M\in \Rep(G)$ let
$M^\beta$ denote the generalized $\beta$-eigenspace for $C_M$. Then $M=\bigoplus_{\beta}M^\beta$. Let $\operatorname{Rep}(G)^\beta$
denote the full subcategory of $\Rep(G)$ consisting of all modules $M$ with $M=M^\beta$. There are no Hom's or extensions
between modules lying in different subcategories $\Rep(G)^\beta$. Another way to phrase this: the category
$\Rep(G)$ splits into the direct sum $\bigoplus_{\beta\in \F_p}\Rep(G)^\beta$.

Pick  a module $M\in \Rep(G)^\beta$. The formula expressing $X$ via $C$ shows that $C$ acts on $F_\alpha(M)$ with a single eigenvalue equal to $\beta+n+2\alpha$. So $F_\alpha(M)$ is the projection of $F(M)$ to $\Rep(G)^{\beta+n+2\alpha}$.

The situation with $E$ is similar. Let $\bar{X}_M$ denote the endomorphism of $E(M)$ induced by $X$. An analog of
Proposition \ref{Prop:casimir_filtr} holds, $\bar{X}_{\Delta(\lambda)}$ acts on $\Delta(\lambda-\epsilon_i)$
by $i-\lambda_i$. We will define $E_\alpha(M)$ as the generalized eigenspace of $\bar{X}_M$ with eigenvalue
$-n-\alpha$. The reason for this choice is that the functors $E_\alpha,F_\alpha$ are biadjoint. This follows
from the projection description of $E_\alpha$: for $M\in \Rep(G)^\beta$, the module $E_\alpha(M)$
is the projection of $E(M)$ to $\Rep(G)^{\beta-n-2\alpha}$. Thanks to this, the biadjointness of $E_\alpha, F_\alpha$
follows from the biadjointness of $E,F$.

\subsection{Action on the Grothendieck group}
Let $\Cat$ be an $\F$-linear artinian abelian category (such as $\Rep(G)$). Recall that ``artinian'' means that
all objects have finite length, and $``\F$-linear'' basically means that all Hom's are vector
spaces over $\F$. Then we can define the complexified
Grothendieck group $[\Cat]$ with generators $[M]$ for objects $M\in \Cat$ and relations $[M_2]=[M_1]+[M_3]$
for every exact sequence $0\rightarrow M_1\rightarrow M_2\rightarrow M_3\rightarrow 0$.

By the definition of $[\Cat]$, this vector space has a basis $[L], L\in \operatorname{Irr}(\Cat)$. However, for $\Cat=\Rep(G)$
we can take a different basis, a so called standard basis $[\Delta(\lambda)], \lambda\in \X^+$. This is a basis
because $[\Delta(\lambda):L(\lambda)]=1$ and $[\Delta(\lambda):L(\mu)]>0\Rightarrow \mu<\lambda$.
Proposition \ref{Prop:casimir_filtr} (and its analog for $E$) allow to compute the operators
$[E_\alpha],[F_\alpha]$. We would like to interpret this computation in a somewhat nicer form that,
in particular, shows that these operators define a representation of the Kac-Moody algebra $\hat{\sl}_p$.

Consider the vector space $\C^{\Z}$ with basis $v_i, i\in \Z$. On this space we introduce operators
$e_\alpha, f_\alpha$ for $ \alpha\in \F_p$ by the following formula:
\begin{equation}
f_\alpha v_i=\begin{cases} &v_{i+1}, i\equiv \alpha\mod p,\\
&0, \text{ else}.\end{cases}\quad e_\alpha v_{i+1}=\begin{cases}v_i, i\equiv \alpha\mod p,\\ 0, \text{ else}.\end{cases}
\end{equation}
These operators define a representation of $\hat{\sl}_p$ on $\C^{\Z}$. Recall that the algebra $\hat{\sl}_p$ (this is the Kac-Moody
algebra associated to the cyclic graph)
is defined  as follows. The generators are $e_\alpha, h_\alpha, f_\alpha, \alpha\in \F_p,$ and the relations are as follows:
\begin{align*}
&h_\alpha=[e_\alpha,f_\alpha], [h_\alpha,e_\alpha]=2e_\alpha, [h_\alpha,f_\alpha]=-2f_\alpha,\\
&[h_\alpha, e_\beta]=-e_\beta, [h_\alpha,f_\beta]=f_\beta, \quad \beta-\alpha=\pm 1,\\
& [h_\alpha,e_\beta]=[h_\alpha,f_\beta]=0, \quad \alpha-\beta\not\in \{-1,0,1\},\\
& [e_\alpha,f_\beta]=0, \quad \alpha\neq \beta,\\
&[e_\alpha,[e_\alpha,e_\beta]]=0, [f_\alpha,[f_\alpha,f_\beta]]=0, \quad \alpha-\beta=\pm 1,\\
&[e_\alpha, e_\beta]=[f_\alpha,f_\beta]=0, \quad \alpha-\beta\neq \{-1,0,1\}.
\end{align*}

It is straightforward to check that the operators $e_\alpha,f_\alpha$ (and $h_\alpha:=[e_\alpha,f_\alpha]$ on $\C^\Z$)
do satisfy the Kac-Moody relations.

Of course, the representation of $\hat{\sl}_p$ on $\C^{\Z}$ induces a representation on $\bigwedge^n \C^\Z$.
Proposition \ref{Prop:casimir_filtr} and its analogue for $E$  imply the following.

\begin{Prop}\label{Prop:K_group}
The spaces $[\Cat]$ and $\bigwedge^n \C^{\Z}$ are isomorphic via $[\Delta(\lambda)]\mapsto \bigwedge_{i=1}^n v_{\lambda_i+1-i}$.
This isomorphism intertwines $[E_\alpha],[F_\alpha]$ and $e_\alpha,f_\alpha$.
\end{Prop}

We remark that $\bigwedge^n \C^{\Z}$ is a {\it level $0$} representation of $\hat{\sl}_p$ meaning that
$\sum_{\alpha}h_\alpha$ acts by $0$. Also this representation is {\it integrable} meaning that
all $e_\alpha,f_\alpha$ act by locally nilpotent endomorphisms. However, this representation is very far
from being highest weight: there are finitely many weights and all weight spaces are infinite dimensional.
More precisely, the weight spaces are parameterized by unordered $n$-tuples of elements of $\F_p$:
the space corresponding to the $n$-tuple $(\alpha_1,\ldots,\alpha_n)$ is spanned by the wedges
$v_{\beta_1}\wedge\ldots \wedge v_{\beta_n}$ such that unordered $n$-tuples $(\alpha_1,\ldots,\alpha_n)$
and $(\beta_1,\ldots,\beta_n)$ coincide. In particular, the classes of Weyl modules are weight vectors.

Below we will see that the ``stable'' categories of polynomial representations ``categorify'' a level 1 highest weight
representation (a {\it Fock space}).

\section{Categorical actions}\label{S_cat_act}
A categorical action of $\hat{\sl}_p$ should consist of  a category together with some additional data.
We have seen most of these data: two endo-functors $E$ and $F$ together with a functor endomorphism $X$
of $F$ (in the sequel we are going to view the tensor Casimir $X$ as an endomorphism of $F$). This is
still not enough, we also need a functor endomorphism $T$ of $F^2(\bullet)=V\otimes V\otimes \bullet$. We are going
to explain how $T$ looks like in our example of the category $\operatorname{Rep}(G)$ in the first subsection.
Then we will give a definition of a categorical $\hat{\sl}_p$-action. We will finish by sketching an application
that to some extent explains the necessity of considering $X,T$. This application is the original motivation
of Chuang and Rouquier for introducing categorical $\sl_2$-actions: their goal was to construct certain derived equivalences.

\subsection{Endomorphisms of $F^N$}\label{SS_endo}
We consider the category $\operatorname{Rep}(G)$ equipped with the endo-functors $E,F$. We are going to establish
a homomorphism from the degenerate affine Hecke algebra $\mathcal{H}_N$ (to be defined below) to the algebra $\operatorname{End}(F^N)$
of endomorphisms of $F^N$.

First, consider the case $N=2$. Let us present three endomorphisms of $F^2$. First, we have an endomorphism $1X$, it is defined
by $(1X)_M=1_V\otimes X_M$, i.e.,
\begin{equation}\label{eq:1X}(1X)_M(v_1\otimes v_2\otimes m)=v_1\otimes X_M(v_2\otimes m)=\sum_{i,j=1}^n v_1\otimes E_{ij}v_2\otimes E_{ji}m.\end{equation}
Next, we have an endomorphism $X1$ given by $(X1)_M=X_{V\otimes M}$, i.e.,
\begin{equation}\label{eq:X1}(X1)_M(v_1\otimes v_2\otimes m)=\sum_{i,j=1}^n E_{ij}v_1\otimes E_{ij}(v_2\otimes m).\end{equation}
Finally, we have a very naive endomorphism $T$ that just switches the two copies of $V$: $T_M(v_1\otimes v_2\otimes m)=v_2\otimes v_1\otimes m$.

We are going to find some relations between $1X,X1,T$. Obviously, $T^2=1$. Next, $X1,1X$ commute. This is because these
are functor morphisms that act on the different copies of $F$ in $F^2=FF$. Indeed, a functor endomorphism $\varphi$
of a functor $F$, by definition, has the following property: for any object $X$ and its endomorphism $\psi$, the endomorphisms
$\varphi_X$ and $F(\psi)$ commute. We need to apply this to $X:=V\otimes M, \varphi:=X, \psi:=X_M$.

The most interesting relation is
$T(X1)-(1X)T=1$. Let us check it.
\begin{align*}T(X1)_M(v_1\otimes v_2\otimes m)&=T\sum_{i,j=1}^n E_{ij}v_1\otimes E_{ji}(v_2\otimes m)=\\&=T(\sum_{i,j=1}^n E_{ij}v_1\otimes E_{ji}v_2\otimes m+E_{ij}v_1\otimes v_2\otimes E_{ji}m)=\\&=\sum_{i,j=1}^n E_{ji}v_2\otimes E_{ij}v_1\otimes m+\sum_{i,j=1}^n v_2\otimes E_{ij}v_1\otimes E_{ji}m,\\
(1X)T_M(v_1\otimes v_2\otimes m)&=(1X)_M(v_2\otimes v_1\otimes m)=\\
&=\sum_{i,j=1}^n v_2\otimes E_{ij}v_1\otimes E_{ji}m.
\end{align*}
So $(T(X1)-(1X)T)_M (v_1\otimes v_2\otimes m)=\sum_{i,j=1}^n E_{ij}v_2\otimes E_{ji}v_1\otimes m$. One can easily check
on the elements of a natural basis in $V=\F^n$ that $\sum_{i,j=1}^n E_{ij}v_2\otimes E_{ji}v_1=v_1\otimes v_2$. So  we have checked
that $T(X1)-(1X)T=1$.

Now consider the general case. We have endomorphisms $X_i:=1^{N-i}X1^{i-1}, i=1,\ldots,N,$ and $T_i=1^{N-i-1}T 1^{i-1}, i=1,\ldots,N-1$.
They satisfy the following relations:
\begin{align*}
& X_i X_j=X_j X_i, \\
& T_i^2=1,\\
&T_i T_j=T_j T_i, \quad |i-j|>1,\\
&T_i T_{i+1}T_i=T_{i+1}T_i T_{i+1},\\
&T_i X_{i+1}- X_iT_i=1,\\
&T_i X_j=X_j T_i, \quad j-i\neq 0,1.
\end{align*}
The algebra generated by $X_1,\ldots,X_N, T_1,\ldots, T_{N-1}$ modulo the relations above
is called the {\it degenerate affine Hecke algebra}, we will denote it by $\mathcal{H}_N$.

\subsection{Definition of a categorical action}
Now we are ready to define a categorical $\hat{\sl}_p$-action. Let $\Cat$ be an
$\F$-linear artinian abelian category. For example, $\Rep(G)$ is such a category.
A categorical $\hat{\sl}_p$-action is the data $(E,F,X,T)$, where $E,F$ are endofunctors of $\Cat$ and $ X\in \End(F), T\in \End(F^2)$,
subject to the following axioms:
\begin{itemize}
\item[(1)] $E,F$ are biadjoint (and, in particular, exact) functors.
\item[(2)] We have the decomposition $F=\bigoplus_{\alpha\in \F_p}F_\alpha$, where $F_\alpha$ is the generalized eigen-subfunctor of $F$
with eigenvalue $\alpha$ with respect to $X$. This automatically yields the decomposition $E=\bigoplus_{\alpha\in \F_p} E_\alpha$,
where $E_\alpha$ are left adjoint to $F_\alpha$.
\item[(3)] The functors $F_\alpha$ and $E_\alpha, \alpha\in \F_p,$ define an integrable representation of $\hat{\sl}_p$
on the complexified Grothendieck group $[\Cat]$.
\item[(4)] The classes of simple objects in $\Cat$ are weight vectors for (the Cartan subalgebra of) $\hat{\sl}_p$.
\item[(5)] We have the equalities $T(X1)-(1X)T=1, T^2=1$ in $\End(F^2)$ and $ (T1)(1T)(T1)=(1T)(T1)(1T)$
in $\End(F^3)$.
\end{itemize}

For $\Cat=\Rep(G)$ we have checked all these axioms but (4).  What we have checked is that the classes of
Weyl modules $\Delta(\lambda)$ are weight vectors. Therefore to check (4) it is enough to check that
$[\Delta(\lambda):L(\mu)]\neq 0$ implies that $[\Delta(\lambda)],[\Delta(\mu)]$ are in the same
weight space. This is a classical fact called the {\it weak linkage principle}, see \cite[6.17]{Jantzen}. In more detail,
one can show that the Serre subcategory spanned by $\Delta(\lambda)$ in a given weight space is the
(generalized) eigen-subcategory with respect to the action of $U(\g)^G$ (compare with the subcategories
$\Rep(G)_\beta$ considered above).

According to \cite[Proposition 5.5]{CR}, (4) implies that the category
$\Cat$ splits into the direct sum $\Cat=\sum_{\gamma}\Cat_\gamma$, where the summation is
taken over the set of weights of $[\Cat]$ and  $[\Cat_\gamma]$ coincides
with the weight space $V_\gamma$.

Let us make several other remarks regarding this definition.

First, $X,T$ induce endomorphisms of $E$ and $E^2$,
respectively. In more detail, since $F$ is left adjoint to $E$, there are
functor morphisms $\eta:\operatorname{Id}\rightarrow EF, \zeta:FE\rightarrow \operatorname{Id}$. Then $X$ defines
an endomorphism of $E$ as follows: $E\xrightarrow{\eta1_E} EFE\xrightarrow{1_EX 1_E} EFE\rightarrow{1_E\zeta} E$.
This description immediately implies that endomorphisms of $\Hom(E X,Y)=\Hom(X, F Y)$
induced by $X$ viewed as an element of $\End(F)$ or of $\End(E)$ are the same.
In our example, the endomorphism of $E$ denoted by $X$ will be given by the element
$-n-\sum_{i,j=1}^n E_{ij}\otimes E_{ji}\in [U(\g)\otimes U(\g)]^G$. This formula
is suggested by the definition of the functors $E_\alpha$ and to check it formally is an exercise.

In particular, $E=\bigoplus_\alpha E_\alpha$ is the eigen-decomposition with respect to $X$.
We remark that similarly to the case of $\Cat=\Rep(G)$, in the general case, $E_\alpha$
is also a right adjoint to $F_\alpha$, thanks to the decomposition $\Cat=\bigoplus_\gamma \Cat_\gamma$.

Second, we can define elements $X_i, i=1,\ldots,N, T_i, i=1,\ldots,N-1$ as before, they produce an algebra
homomorphism $\mathcal{H}_N\rightarrow \End(F^N)$. We also have a representation of $\mathcal{H}_N$ in $\operatorname{End}(E^N)$ given
in the completely analogous way.

We also would like to point out that the definition of a categorical action can be generalized to other
Lie algebras of type $A$. For example, for $\sl_2$ (the first case considered by Chuang and Rouquier),
one just should replace (2) with the condition that $X_M-\alpha$ is a nilpotent for some $\alpha\in \F_p$
and any $M\in \Cat$. Clearly, a categorical  $\hat{\sl}_p$-action is a collection of
categorical $\sl_2$-actions subject to some additional compatibility conditions.

\subsection{Rickard complex}
Let us briefly explain the original motivation of Chuang and Rouquier to introduce categorical $\sl_2$-actions
that shows the importance of the Hecke action on $F^N$. Consider a categorical $\sl_2$-action that yields a decomposition
$\Cat=\bigoplus_d \Cat_d$ into the weight subcategories. The goal of Chuang and Rouquier was to produce a derived
equivalence between $\Cat_{-d}$ and $\Cat_d$. This implied a proof of Broue's abelian defect group conjecture
for symmetric groups (it does not matter what the conjecture is about).

The weight spaces $[\Cat_d]=[\Cat]_d,[\Cat_{-d}]=[\Cat]_{-d}$ are isomorphic,
an isomorphism $[\Cat]_{-d}\rightarrow [\Cat]_{d}$ is given by the non-trivial Weyl group
element. This isomorphism can be expressed via the operators $e$ and $f$ as follows:
$$\theta=\bigoplus_{k=0}^{+\infty} (-1)^k\frac{e^{k+d}}{(k+d)!}\frac{f^k}{k!}.$$

When we try to write a functor (or a complex of functors) ``categorifying'' this expression,
we run into a problem. It is easy to divide a linear operator on a vector space by a nonzero
scalar. But one cannot, in general, divide a functor. We can only divide a functor, say $F^d$,
by $d!$, if $F^d$ is isomorphic to the sum of $d!$ copies of another functor, say $F^{(d)}$ (then,
of course, $\frac{F^d}{d!}=F^{(d)}$). But now $F^d(M)$ is a module over $\mathcal{H}_d$ and the element
$X_1\in \mathcal{H}_d$ acts with a single eigenvalue. The structure of such modules is well-understood,
in particular they decompose into the sum of $d!$ summands (and this decomposition is functorial with respect to
$M$). This gives rise to the divided power functors $F^{(d)}$ -- and also to $E^{(d)}$.

The next task is to form a complex
$$ \ldots\rightarrow E^{(d+2)}F^{(2)}\rightarrow E^{(d+1)}F\rightarrow E^{(d)}\rightarrow 0,$$
this complex of functors will be a desired equivalence. The differentials in the complex are again constructed
using the representation theory of degenerate affine Hecke algebras.

\section{Crystals}\label{S_cryst}
\subsection{Crystals of $\g$-modules}
Let $\g$ be a Kac-Moody algebra (we will be interested in the case of $\hat{\sl}_p$)
and let $e_i,f_i$ denote the generators, where $i$ is an element of some indexing set $I$. Let $\g_i$
be the subalgebra of $\g$ generated by $e_i,f_i$, it is, of course, isomorphic to $\sl_2$.

A crystal structure is a combinatorial shadow of a $\g$-module structure. Crystals were defined by Kashiwara
using quantum groups. We will follows an approach of Berenstein and Kazhdan, \cite{BK}, that
define crystals without quantum groups.

Let $N$ be an integrable $\g$-module. For an element $n\in N$ define $d_i(n)$ as
the maximal dimension of the irreducible $\sl_2$-submodule in $U(\g_i)n$. Let $N_i(< d)$
denote the span of all vectors $n\in N$ with $d_i(n)< d$.

A basis $\mathcal{B}$ of $N$ is called {\it perfect} if it consists of weight vectors and  there are maps $\tilde{e}_i,\tilde{f}_i:\mathcal{B}\rightarrow
\mathcal{B}\cup\{0\}$ with the property that
$$e_i b\in \C^\times  (\tilde{e}_i b) + N_i(<d_i(b)), \qquad f_i b\in \C^\times (\tilde{f}_i b)+ N_i(<d_i(b)), \quad \forall i\in I.$$

The set $\mathcal{B}$ with a collection of maps $\tilde{e}_i,\tilde{f}_i:\mathcal{B}\rightarrow \mathcal{B}\sqcup\{0\}$
is called a {\it crystal} of $N$ (to get the definition of an abstract crystal one should impose certain axioms on these maps but we
are not going to do this, one of the axioms is that if $\tilde{e}_i b=b'\neq 0$, then $\tilde{f}_i b'=b$). As a subset
of $N$, the set $\mathcal{B}$ is not defined uniquely but the crystal is defined uniquely up to an isomorphism
(and if $N$ is an irreducible highest weight module, then there is a unique automorphism of its crystal).
For the proofs the reader is referred to \cite{BK}.  We remark that we are not going to discuss the questions of
existence: in the cases of interest for us a perfect basis always exists.

\subsection{Crystal of a categorical action}
The reason why we are interested in crystals is that any categorical $\hat{\sl}_p$-action on $\Cat$ gives rise to
a canonical crystal structure on the set $\operatorname{Irr}(\Cat)$ of simple objects in $\Cat$. Moreover, we will
see that the classes of simples form a perfect basis in $[\Cat]$.

To a nonzero object $M\in \Cat$ we can assign its {\it head} $\operatorname{head}(M)$, the maximal semisimple quotient, and
its socle, $\operatorname{soc}(M)$, the maximal semisimple subobject.

Now suppose $\Cat$ is equipped with a categorical $\sl_2$-action, with functors $E,F$. For an object $M\in \Cat$
we set $d(M):=d([M])$, in the notation of the previous subsection. So $d(M)$ equals to $d_F+d_E+1$, where
$d_F$ (resp., $d_E$) is the maximal number such that $F^{d_F}M\neq 0$ (resp., $E^{d_E}M\neq 0$). For a simple object $L$,
the objects $EL, FL$ are not simple, in general. However, the following result, due to Chuang and Rouquier, holds.

\begin{Prop}[\cite{CR}, Proposition 5.20]\label{Prop:crystal}
Suppose $EL\neq 0$. The head and the socle of $EL$ are isomorphic simple objects. Denote this object by
$\tilde{e}L$. All irreducible constituents $L'$ of $EL$ different from $\tilde{e}L$  satisfy $d(L')<d(L)$.
The analogous results holds for $F$.
\end{Prop}

For a categorical $\hat{\sl}_p$-action, the previous proposition defines the crystal operators
$\tilde{e}_\alpha, \tilde{f}_\alpha:\operatorname{Irr}(\Cat)\rightarrow \operatorname{Irr}(\Cat)\sqcup\{0\}$.
The proposition also shows that the basis in $[\Cat]$ consisting of the classes of simples is perfect.

\subsection{Computation for $\operatorname{Rep}(G)$}
Now the question is: how to compute the crystal structure. In many cases this question does not make much
sense: we need a classification of $\operatorname{Irr}(\Cat)$ not via the crystal structure but in some other
terms. But often, for example, for $\Cat=\Rep(G)$ we do have such  a description: $\operatorname{Irr}(\Cat)$
is identified with $\X^+$ and we need to  compute the maps $\tilde{e}_\alpha, \tilde{f}_\alpha:\X^+
\rightarrow \X^+\sqcup\{0\}$.

There is an explicit combinatorial procedure for this first discovered by Brundan and Kleshchev, \cite{BK_gen_lin},
and then rediscovered in \cite{cryst} in a more general context. The procedure, producing a combinatorial
crystal structure on $\X^+$, is in three steps. First,
from $\lambda\in \X^+$ and $\alpha\in \F_p$ we produce a sequence of $+$'s and $-$'s, called
the {\it $\alpha$-signature} of $\lambda$. Second, we perform a certain reduction procedure getting a
{\it reduced signature}. Finally, looking at the reduced signature, we define $\tilde{e}_\alpha \lambda,
\tilde{f}_\alpha \lambda$.

Let us explain how to produce the $\alpha$-signature. Recall that we defined the subsets $I^{\pm i}_\lambda\subset
\{1,\ldots,n\}$ of all indexes $i$ such that $\lambda\pm\epsilon_i\in \X^+$.
We read the numbers $\lambda_i$ from left to right, starting from $\lambda_1$. We write a $+$ if we encounter
$i\in I^+_\lambda$ with $\lambda_i+1-i\equiv \alpha \mod p$. We write a $-$ if we encounter $i\in I^-_\lambda$
with $\lambda_i+2-i\equiv \alpha\mod p$. For example, let $p=5, \alpha=2$ and
$$\lambda=(18,16,15,15,12,7,7,5,0,-4,-8,-12,-15,-19).$$ For convenience
let us write the $n$-tuple $(\lambda_i+1-i)_{i=1}^n$: it equals $$(18,15,13,12,8,2,1,-2,-8,-13,-18,-23, -27,-32).$$
The entries on positions $2,3,4,7$ do not contribute to the signature, the other entries give the sequence
$--+-++++--$ (in fact, a signature is more than just a collection of $+$'s and $-$'s, with each element we associate
the index of the entry producing the element).

Proceed to defining the reduced signature. We will consequently remove the consecutive pairs $-+$ in the $\alpha$-signature
(leaving ``empty places'').
We do keep removing  until possible, so we finish when
all $+$'s that remain are located to the left of all $-$'s. In our example, we can remove
positions 2 and 3, then 4 and 5. After that we still have one more removal, 1 and 6, and then we are done. What remains,
$++--$, is the reduced signature.

Now the maps $\tilde{e}_\alpha, \tilde{f}_\alpha$ are constructed as follows. To define $\tilde{f}_\alpha$ take
the right-most $+$ in the reduced signature. Let $i\in I^+_\lambda$ be the corresponding index. Then $\tilde{f}_\alpha\lambda$
is obtained from $\lambda$ by increasing $\lambda_i$ by $1$. Similarly, to define $\tilde{e}_\alpha \lambda$, take $i$
corresponding to the left-most $-$, and decrease the corresponding $\lambda_i$ by $1$. In our example, the right-most
$+$ corresponds to $i=12$ and $\tilde{f}_\alpha\lambda=(18,16,15,15,12,7,7,5,0,-4,-8,-11,-15,-19)$. The left-most $-$
corresponds to $i=13$ and so $\tilde{e}_\alpha\lambda=(18,16,15,15,12,7,7,5,0,-4,-8,-12,-16,-19)$.

The result, due to Brundan and Kleshchev, is that, under the identification $\operatorname{Irr}(\Rep(G))\cong \X^+$,
the crystal operators we have just constructed are the crystal operators on $\operatorname{Irr}(\Rep(G))$ defined
using the categorical action. What this gives is the complete description of, say, irreducible submodules
of $V\otimes L(\lambda)$.

The description of the crystal may seem bizarre. In fact, it is quite natural (in a way, this is the only structure
one may get) and also holds in a larger generality: an analogous description works for any highest weight
category equipped with a categorical $\sl_2$-action modulo certain compatibility conditions relating the
highest weight structure and the categorical action. In the remainder of the subsection we will try to argue
that the description is natural.

We need to determine the heads of $F_\alpha L(\lambda)$ and $E_\alpha L(\lambda)$. But we have  surjections
$F_\alpha \Delta(\lambda)\twoheadrightarrow F_\alpha L(\lambda), E_\alpha \Delta(\lambda)\rightarrow E_\alpha L(\lambda)$.
Because of this, the head of $F_\alpha L(\lambda)$ is contained in the head of $F_\alpha \Delta(\lambda)$.
Recall that we have a filtration on $F_\alpha\Delta(\lambda)$ whose consecutive quotients are Weyl modules.
So the head of $F_\alpha\Delta(\lambda)$ consists of the simple quotients of some of these Weyl modules.
On the level of signatures, the highest weights of subquotients correspond to replacing a $+$ in the signature
with a $-$. We just need to locate that $+$.

Let us restate the combinatorial recipe. For a signature $t$ we define its weight $\operatorname{wt}(t)$
to be equal the number of $-$'s minus the number of $+$'s. We order the signatures of given length
and weight  in the inverse lexicographic order assuming that $->+$. I.e., for signatures $t=(t_1,\ldots,t_n),
s=(s_1,\ldots,s_n)$ we write $s<t$ if there is $m\in \{1,2,\ldots,n\}$ such that $t_{m+1}=s_{m+1},\ldots t_n=s_n$
and $t_m=-, s_m=+$. In particular, the largest signature of given length and weight is reduced.

One can check that our combinatorial recipe (say for $\tilde{f}_\alpha$) can be restated as follows.
We list elements $t^1,\ldots,t^k,\ldots$ of given weight in the increasing order. Then one can check
that $\tilde{f}_\alpha t^k$ is either the largest signature obtained from $t$ by replacing a $+$ with a
$-$ that is different from $\tilde{f}_\alpha t^j, j=1,\ldots,k-1$ or $0$ if no such exists.

The proof given in \cite{cryst} builds on this observation and is a pretty formal game.

\section{Polynomial case}\label{S_pol}
\subsection{Schur-Weyl duality}
Consider the characteristic $0$ case first. We say that a simple $G$-module
$\Delta(\lambda)$ is {\it polynomial of degree $d$} if its matrix coefficients
are degree $d$ homogeneous polynomials of the matrix entries. In terms of $\lambda$,
this means that $\lambda_n\geqslant 0$ and $\lambda_1+\ldots+\lambda_n=d$, i.e.,
$\lambda$ is a partition of $d$. The set of all partitions of $d$
will be denoted by $\mathcal{P}(d)$. One thing to notice here is that the labeling set of
the degree $d$ representation is the same for all $n\geqslant d$.
It is not difficult to see that $L(\lambda)$ is polynomial of degree $d$
if and only if it is polynomial of degree $d$ as a representation of $T$,
meaning that all weights $\mu$ of $L(\lambda)$ satisfy $\mu_1,\ldots,\mu_n\geqslant 0, \mu_1+\ldots+\mu_n=d$.

Yet one more equivalent definition: $L(\lambda)$ is polynomial of degree $d$ if it is a direct summand in
$V^{\otimes d}$, where $V=\F^n$ is the tautological representation.
In fact, on $V^{\otimes d}$ we have an action of the symmetric group $S_d$
permuting the factors, this action commutes with $G$. For $n\geqslant d$
we have the {\it Schur-Weyl duality}: $V^{\otimes n}=\bigoplus_\lambda \Delta(\lambda)\otimes S_\lambda$,
where $S_\lambda$ is the (simple) Specht $S_d$-module and the summation is over
all partitions of $d$.

Not surprisingly, in characteristic $p$, the situation again becomes more subtle
due to the absence of complete reducibility.
First, we define the category of polynomial representations of degree $d$, $\operatorname{Rep}^d(\GL_n)$,
as the full subcategory in $\operatorname{Rep}(\GL_n)$ consisting of all modules whose $T$-weights
$\mu$ satisfy $\mu_1,\ldots,\mu_n\geqslant 0, \mu_1+\ldots+\mu_n=d$. From this definition, we see that $\Delta(\lambda)$
lies in $\operatorname{Rep}^d(\GL_n)$ if and only if $\lambda\in \mathcal{P}(d)$. Indeed, the weights
of $\Delta(\lambda)$ are the same as  in characteristic $0$. From here it is easy
to see that a $G$-module lies in $\operatorname{Rep}^d(\GL_n)$ if and only if all its simple constituents
are of the form $L(\lambda), \lambda\in \mathcal{P}(d)$.

Now let us discuss the dependence of $\operatorname{Red}^d(\GL_n)$ on $n$. What we have seen in the characteristic
$0$ story is that the category is independent of $n$ as long as $n\geqslant d$. But there this was true for a very simple
reason: the categories are semisimple and we just have a bijection between the simples. But the equivalence
result is still true in characteristic $p$. Namely, we have a functor $\tau_{n}^{n+1}:\operatorname{Rep}^d(\GL_{n+1})\rightarrow
\operatorname{Rep}^d(\GL_n)$ that sends a module $M$ to its invariants for the action of the one-dimensional
subtorus $S=\{\operatorname{diag}(1,\ldots,1,t)\in \GL_{n+1}\}$. It is not difficult to check that this functor is exact and sends $\Delta(\lambda,0)$ to $\Delta(\lambda)$.

Moreover, let us show that this functor is an equivalence. Let $\mathfrak{p}$ be the parabolic subalgebra of $\gl_{n+1}$ spanned by $E_{ij}$
with $i\leqslant n$ or $i=j=n+1$ and $\dot{U}(\mathfrak{p})\subset \dot{U}(\gl_{n+1})$ be the corresponding
hyperalgebra. Consider a functor
$\psi:\dot{U}(\gl_n)$-$\operatorname{mod}\rightarrow \dot{U}(\gl_{n+1})$-$\operatorname{mod}$ that sends
$N$ to the quotient of $\dot{U}(\gl_{n+1})\otimes_{\dot{U}(\mathfrak{p})} N$ by the maximal submodule
that does not intersect the $S$-weight space of maximal weight.  That weight space coincides with
the (actually, isomorphic) image of  $N$ in $\dot{U}(\gl_{n+1})\otimes_{\dot{U}(\mathfrak{p})} N$.
 One can show that $\tau_{n}^{n+1}\circ \psi$ is the identity. This shows that $\tau_n^{n+1}$
is a quotient functor. But the labels of the simples in $\Rep^d(\GL_{n+1})$
all have form $(\lambda,0)$. The previous paragraph shows that $\tau_n^{n+1}(L(\lambda,0))=L(\lambda)$
and, in particular, $\tau_n^{n+1}$ does not kill any simple. A quotient functor that does not kill any simple
is automatically an equivalence.


So we can consider
the stable category $\operatorname{Rep}^d(\GL):=\operatorname{Rep}^d(\GL_n)$ with $n\geqslant d$.

The Schur-Weyl duality still holds in some form, and, again, it is a functor rather than a bijection.
Namely, for $n\geqslant d$, we can consider the {\it Schur functor} $\mathcal{S}: \operatorname{Rep}^d(\GL_n)\rightarrow S_d$-$\operatorname{mod}$ given by $\mathcal{S}(M)=\operatorname{Hom}_{\GL_n}(V^{\otimes d},M)$, where $V$ is the
tautological $\GL_n$-module $\F^n$. This functor is exact. This is because $V^{\otimes d}$ is a projective object in
$\operatorname{Rep}^d(\GL_n)$ (but, in general, it  is not projective in the whole category
$\operatorname{Rep}(\GL_n)$). The functor satisfies a one-sided  double centralizer property that can be
stated as $\operatorname{Hom}_{\GL_n}(P_1,P_2)=\operatorname{Hom}_{S_d}(\mathcal{S}(P_1),\mathcal{S}(P_2))$
for any two projective objects in $\operatorname{Rep}^d(\GL_n)$ (the latter category has enough projective,
a projective cover of $\Delta(\lambda)$ in $\operatorname{Rep}^d(\GL_n)$ is the largest quotient
of $P(\lambda)$ lying in the subcategory). One can describe the image of $\Delta(\lambda)$,
this is a so called {\it dual Specht module}. However, let us point out that the functor is
not an equivalence, in general,  it does send some simple objects to $0$.


We write $\operatorname{Pol}(\GL):=\bigoplus_{d=0}^{+\infty} \operatorname{Rep}^d(\GL)$.

\subsection{Categorical action on $\operatorname{Pol}(\GL)$}
Here we are going to introduce a categorical $\hat{\sl}_p$-action on $\operatorname{Pol}(\GL)$.
This action was first introduced by Hong and Yacobi in \cite{HY} in a considerably more technical
fashion.

It is easy to define an analog of the functor $F$ that will map $\operatorname{Rep}^d(\GL_n)$
to $\operatorname{Rep}^{d+1}(\GL_n)$, we can simply use the same formula as before,  $F(M):=V\otimes M$.
In fact, as the following lemma shows, the functor $F$ does not depend on $n$ and hence lifts
to $\operatorname{Pol}(\GL)$.

\begin{Lem}\label{Lem:F_lift}
We have $\tau_{n}^{n+1}\circ F(M)=F\circ \tau_n^{n+1}(M)$ for any $M\in \Rep^d(\GL_{n+1})$.
\end{Lem}
\begin{proof}
Recall that $S$ denotes the one-dimensional subtorus $\{\operatorname{diag}(1,\ldots,1,t)\}$ in $\GL_{n+1}$
so that $\tau^{n+1}_n(M)=M^S$. Clearly, for $V=\F^{n+1}$, the $\GL_n$-module $V^S$ is the tautological
module. So what we need to prove is $(V\otimes M)^S=V^S\otimes M^S$. The right hand side is
included into the left one. On the other hand, the left hand side is spanned by vectors of the form
$v\otimes m$, where $v,m$ are weight vectors with weights, say $\alpha,\beta$, such that $\alpha_{n+1}+\beta_{n+1}=0$.
But $\alpha_{n+1},\beta_{n+1}\geqslant 0$ because both $V,M$ are polynomial representations. So
$\alpha_{n+1}=\beta_{n+1}=0$ and we are done.
\end{proof}

So we have an endo-functor $F$ of $\operatorname{Pol}(\GL)$ with $F(\operatorname{Pol}^d(\GL))=\operatorname{Pol}^{d+1}(\GL)$.
Moreover, we can define a natural transformation $X$ of $F$ in the same way as before. Namely, let $X_n$
denote the tensor Casimir defined for the group $\GL_n$.

\begin{Lem}\label{Lem:tens_Cas}
Let $M\in \operatorname{Rep}^d(\GL_{n+1})$. The restriction of $X_{n+1}$ to $\tau_{n}^{n+1}(V\otimes M)=V^S\otimes M^S$
coincides with $X_n$.
\end{Lem}
\begin{proof}
We have $E_{in+1}v=0$ for $v\in V^S$ and all $i$, this is  a direct computation. Also $E_{in+1}m=0$ for $m\in M^S$ and $i\leqslant n$
because the weight of $E_{in+1}m$ has negative $n+1$th component.
Therefore $\sum_{i,j=1}^{n+1}E_{ij}\otimes E_{ji}(v\otimes m)=\sum_{i,j=1}^{n+1} E_{ij}v\otimes E_{ji}m=
\sum_{i,j=1}^n E_{ij}v\otimes E_{ji}m$ provided $v\in V^S, m\in M^S$ and we are done.
\end{proof}

This lemma shows that the functors $F_\alpha, \alpha\in \F_p$, are well-defined on $\operatorname{Pol}(\GL)$
(meaning that $F_\alpha\circ \tau_n^{n+1}=\tau_n^{n+1}\circ F_\alpha$). Also it is straightforward to
check that $T\circ \tau_n^{n+1}=\tau_n^{n+1}\circ T$.

To define $E$ on $\operatorname{Pol}(G)$  is more complicated. The reason is that the representation  $V^*\otimes M$
does not need to be polynomial even if $M$ is polynomial (the simplest example: $M$ is the trivial representation).

The following lemma plays a crucial role in defining the functors $E_\alpha$ on $\operatorname{Pol}(\GL)$.

\begin{Lem}\label{Lem:def_E}
Let $d< n$ and $\alpha\neq 2-n \mod p$. Then the functor $E_\alpha$ maps the subcategory
$\operatorname{Rep}^d(\GL_n)$ into $\Rep^{d-1}(\GL_n)$.
\end{Lem}
\begin{proof}
It is enough to prove that $E_\alpha L(\lambda)\in \operatorname{Rep}^{d-1}(\GL_n)$ when $L(\lambda)\in \operatorname{Rep}^d(\GL_n)$.
This will follow once we check the analogous claim for $\Delta(\lambda)$. Recall that $V^*\otimes \Delta(\lambda)$ has a filtration
with subsequent quotients $\Delta(\lambda-\epsilon_i)$ for $i\in I^-_\lambda$. The only of these quotients that is not polynomial
is $\Delta(\lambda-\epsilon_n)$ (that is actually a submodule). This quotient lies in $E_{2-n}\Delta(\lambda)$. Our claim follows.
\end{proof}

Thanks to the previous lemma, we can {\it define} $E_\alpha$ on $\operatorname{Pol}(\GL)$ as the left adjoint of $F_\alpha$
-- the lemma shows that this functor exists and is also right adjoint to $F_\alpha$. A different but equivalent construction of the 
functors $E_\alpha$ was given in \cite{HY} using the language of {\it polynomial functors}. 

Now it is straightforward to check that the functors $F$ and $E:=\bigoplus_{\alpha\in \F_p}E_\alpha$ together with
functor endomorphisms $X,T$ define a categorical $\hat{\sl}_p$-action on $\operatorname{Pol}(\GL)$.

\subsection{Grothendieck group}
Passing from the usual category $\operatorname{Rep}(\GL_n)$ to the stable polynomial category $\operatorname{Pol}(\GL)$
may seem artificial. An advantage of the latter category is that its Grothendieck group is much better: it is
a so called {\it Fock space} $\mathcal{F}$ of $\hat{\sl}_p$.

We are going to write weights $\lambda=(\lambda_1,\ldots,\lambda_n)$ as Young diagrams (we use the convention
that the lengthes of rows decrease bottom to top). Recall that the content $\operatorname{cont}(b)$ of a box $b$
lying in the $x$th row and the $y$th column is $y-x$. We say that $b$ is an $\alpha$-box if $\operatorname{cont}(b)\equiv
\alpha\mod p$. A box $b$ lying in $\lambda$ is said to be {\it removable} if $\lambda\setminus \{b\}$ is again
a Young diagram. A box $b$ lying outside $\lambda$ is said to be {\it addable} if $\lambda\cup \{b\}$ is a Young
diagram. Proposition \ref{Prop:casimir_filtr} can be reinterpreted as follows: $F_\alpha\Delta(\lambda)$ has a filtration
with successive quotients $\Delta(\mu)$, where $\mu$ runs over the set of diagrams that can be obtained from
$\lambda$ by adding an $\alpha$-box. Similarly, an analog of Proposition \ref{Prop:casimir_filtr}, together with the
construction of $E_\alpha$'s implies that $E_\alpha\Delta(\lambda)$ has a filtration with successive quotients
of the form $\Delta(\nu)$, where $\nu$ is obtained from $\lambda$ by removing an $\alpha$-box.

The complexified Grothendieck group $[\operatorname{Pol}(\GL)]$ has a basis number by all Young diagrams
and corresponding to the Weyl modules in $\operatorname{Pol}(\GL)$. Again, it is easy to see that
the operators $[E_\alpha],[F_{\alpha}]$ define a representation of $\hat{\sl}_p$ on $[\operatorname{Pol}(\GL)]$.
The representation is clearly integrable but now it is also highest weight (but not irreducible).
The central element $\sum_{\alpha\in \F_p}h_\alpha$ acts by $1$.

\subsection{Relation with the symmetric group categorification}
The Fock space representation of $\hat{\sl}_p$ is not irreducible: to each diagram (whose row lengthes are) divisible
by $p$, there corresponds a singular vector. Being an integrable highest weight representation, $\mathcal{F}$ is completely
reducible. We are interested in the trivial component of $\mathcal{F}$ generated by the basis vector corresponding
to empty multipartition, denote this component by $\mathcal{F}_{\varnothing}$.

It turns out that the Schur functor $\mathcal{S}: \operatorname{Pol}(\GL)\rightarrow \F\mathfrak{S}$-$\operatorname{mod}:=
\bigoplus_{d=0}^{+\infty} \F \mathfrak{S}_d$-$\operatorname{mod}$ ``categorifies'' the projection $\mathcal{F}\rightarrow \mathcal{F}_{\varnothing}$ (meaning that $[\mathcal{S}]$ equals to that projection). This was first proved in  \cite{HY_Schur}.
We are going to propose a less rigorous but more elementary approach to the proof of that fact.

First of all, we need to explain a categorical $\hat{\sl}_p$-action on
$\Cat:=\F\mathfrak{S}$-$\operatorname{mod}$. We can define the endo-functors
$E,F$ as follows. First,  $E:=\bigoplus_{d=0}^{+\infty}\operatorname{Res}_{d-1}^{d}$, where
$\operatorname{Res}_{d-1}^d: \F\mathfrak{S}_d$-$\operatorname{mod}\rightarrow
\F\mathfrak{S}_{d-1}$-$\operatorname{mod}$ is the restriction functor
(we set $\operatorname{Res}_{-1}^0:=0$). Similarly, $F:=\bigoplus_{d=0}^{+\infty}\operatorname{Ind}_{d}^{d+1}$, where
$\operatorname{Ind}_{d}^{d+1}: \F\mathfrak{S}_d$-$\operatorname{mod}\rightarrow
\F\mathfrak{S}_{d+1}$-$\operatorname{mod}$ is the induction functor.

In this situation it is more convenient to define endomorphisms of $E$ and of $E^2$. An endomorphism
$X$ of the restriction functor $\operatorname{Res}_{d-1}^d$ is given by the Jucys-Murphi element
$X_d=\sum_{i=1}^{d-1}(id)$, where $(id)$ denotes the transposition in $\mathfrak{S}_d$ permuting $i$ and $d$. The endomorphism
$T$ of $E^2$ is given by the transposition $(d-1d)$.

To check that $\mathcal{S}$ is a morphism of categorical actions we first need to show that it intertwines
the functors $E,F$.

\begin{Lem}\label{Lem:funct_iso}
There are isomorphisms  $E\mathcal{S}\cong \mathcal{S}E, F\mathcal{S}\cong \mathcal{S}F$.
\end{Lem}
\begin{proof}
Fix $d$ and $n>d+1$. It is enough to establish isomorphisms of functors on $\operatorname{Pol}^d(\GL_n)$.
The first isomorphism is easy: $\mathcal{S}(EM)=\Hom_{\GL_n}(V^{\otimes d-1}, V^*\otimes M)=\Hom_{\GL_n}(V^{\otimes d},M)=E\mathcal{S}(M)$.

The second isomorphism is more complicated. Observe that $\mathcal{S}(FM)=\operatorname{Hom}_{\GL_n}(V^{\otimes d+1},V\otimes M)=
\operatorname{Hom}_{\GL_n}(V^{\otimes d+1}\otimes V^*,M)$. On the other hand, $F\mathcal{S}(M)=\F \mathfrak{S}_{d+1}\otimes_{\F\mathfrak{S}_d}\Hom_{\GL_n}(V^{\otimes d},M)$. We can embed $\Hom_{\GL_n}(V^{\otimes d},M)$
into $\operatorname{Hom}_{\GL_n}(V^{\otimes d+1}\otimes V^*,M)$. Namely, we send $\varphi\in \Hom_{\GL_n}(V^{\otimes d},M)$
to $\varphi\otimes \operatorname{tr}$, where $\operatorname{tr}$ is the natural map $V\otimes V^*\rightarrow \F$.
Clearly, the image of our embedding is $\mathfrak{S}_d$-stable, so we get an $\mathfrak{S}_{d+1}$-equivariant map
$$F\mathcal{S}(M)=\F\mathfrak{S}_{d+1}\otimes_{\F\mathfrak{S}_d}\Hom_{\GL_n}(V^{\otimes d},M)\rightarrow
\operatorname{Hom}_{\GL_n}(V^{\otimes d+1}\otimes V^*,M)=\mathcal{S}(FM)$$
that is clearly functorial in $M$. What we need to show is that this map is an isomorphism.
This is equivalent to showing that there is an epimorphism $V^{\otimes d+1}\otimes V^*\rightarrow (V^{\otimes d})^{\oplus d+1}$
whose kernel admits no nonzero homomorphisms to a polynomial representation. We remark that this is definitely so
when the characteristic of $\F$ is zero. Indeed, in this case, $\mathcal{S}$ is an equivalence of categories.
It intertwines the $E$-functors and, by the adjunction, has to intertwine the $F$-functors.

In characteristic $p$, one can argue as follows. Recall that the $\GL_n$-module $V^{\otimes d+1}\otimes V^*$
admits a filtration whose quotients are Weyl modules. We remark that if $\Delta(\lambda)\in \operatorname{Pol}^d(\GL_n)$
and $\Delta(\lambda')\not\in \operatorname{Pol}^d(\GL_n)$, then $\lambda'\not<\lambda$. It follows that any
module in $\Rep(\GL_n)$ that admits a Weyl filtration has a maximal Weyl filtered quotient belonging to
$\operatorname{Pol}^d(\GL_n)$ and, in the filtration on the kernel, there are no Weyl factors that
belong to $\operatorname{Pol}^d(\GL_n)$. So we need to show that the  maximal degree $d$ polynomial quotient for $V^{\otimes d+1}\otimes V^*$
is $(V^{\otimes d})^{\oplus d+1}$. But this is a property that can be seen on the level of characters and,
since the characters are independent of the characteristic,  we deduce the property in characteristic $p$
from the already known property in characteristic $0$.
\end{proof}

To show that $\mathcal{S}$ is a morphism of categorical actions it remains to prove the following lemma.

\begin{Lem}\label{Lem:transf_compat}
The functor $\mathcal{S}$ respects the functor transformations $X$ and $T$. I.e., under the identification
$E\mathcal{S}\cong \mathcal{S}E$ the two transformations of this functor coming from the $X$'s coincide
(and the similar claim for the endomorphism $T$ of $E^2\mathcal{S}\cong \mathcal{S}E^2$).
\end{Lem}
We will prove the lemma in the (harder) case of $X$. The proof for $T$ is left to the reader.
\begin{proof}
Let $X^{\mathfrak{S}}, X^{\GL}$ denote the transformations of $E\mathcal{S}\cong \mathcal{S}E$
coming from  the categories $\F\mathfrak{S}$-$\operatorname{mod}$ and $\operatorname{Pol}(\GL)$.
By definition, $X^{\mathfrak{S}}_M$ is the endomorphism of $\Hom_{\GL_n}(V^{\otimes d},M)$
given by $$X_M^{\mathfrak{S}}\varphi(v_1\otimes\ldots\otimes v_{d-1}\otimes v_d):=\sum_{i=1}^{d-1}\varphi(v_1\otimes\ldots \otimes v_{i-1}\otimes v_d\otimes v_{i+1}\otimes\ldots\otimes v_i).$$
On the other hand, $X^{\GL}_M$ is the endomorphism of $\Hom_{\GL_n}(V^{\otimes d-1},V^*\otimes M)$ given by
$$X_M^{\GL}\psi=-n\psi -(\sum_{i,j} E_{ij}\otimes E_{ji})\circ\psi.$$
The spaces $\Hom_{\GL_n}(V^{\otimes d},M)$ and $\Hom_{\GL_n}(V^{\otimes d-1},V^*\otimes M)$ are identified by
an isomorphism $\iota$ defined as follows. Take $\varphi\in \Hom_{\GL_n}(V^{\otimes d},M)$. Consider the map
$\varphi\otimes \operatorname{id}_{V^*}:V^{\otimes d}\otimes V^*\rightarrow V^*\otimes M$. Then $\iota(\varphi)=\varphi\circ E_d$,
where $E_d$ is the map $t\mapsto t\otimes E$, here $t\in V^{\otimes d-1}$ and $E$ is the element of $V\otimes V^*$
corresponding to the identity map. What we need to prove is: $X_M^{\GL}\iota(\varphi)=\iota(X_M^{\mathfrak{S}}\varphi)$.

Let us rewrite $X_M^{\mathfrak{S}}\varphi$. As we have seen in Subsection \ref{SS_endo},
$v_2\otimes v_1=\sum_{i,j}E_{ij}v_1\otimes E_{ji}v_2$. So the map
$$v_1\otimes\ldots v_d\mapsto \sum_{i=1}^{d-1}v_1\otimes\ldots \otimes v_{i-1}\otimes v_d\otimes v_{i+1}\otimes\ldots\otimes v_i $$
is nothing else but
$$v_1\otimes\ldots\otimes v_d\mapsto \sum_{i,j}E_{ij}(v_1\otimes\ldots\otimes v_{d-1})\otimes E_{ji}v_d.$$
Recall that $\sum_{i,j=1}^n E_{ij}\otimes E_{ji}=\frac{1}{2}(\delta(C)-C\otimes 1-1\otimes C)$, where
$C$ is the Casimir element. So
\begin{align*}
&X^{\mathfrak{S}}_M\varphi(v_1\otimes\ldots\otimes v_d)=\frac{1}{2}\varphi\left(C_{V^{\otimes d}}(v_1\otimes\ldots\otimes v_d)\right)-
\frac{1}{2}\varphi\left(C_{V^{\otimes d-1}}(v_1\otimes\ldots\otimes v_{d-1})\otimes v_d\right)\\
&-\frac{1}{2}\varphi(v_1\otimes\ldots\otimes v_{d-1}\otimes C_Vv_d)=\frac{1}{2}C_M\varphi(v_1\otimes\ldots\otimes v_d)-\frac{1}{2}\varphi\left(C_{V^{\otimes d-1}}(v_1\otimes\ldots\otimes v_{d-1})\otimes v_d\right)\\
&-\frac{n}{2}\varphi(v_1\otimes\ldots\otimes v_d).
\end{align*}
Here we used the fact that $\varphi$ is $\GL_n$- and hence $U(\g)$-linear and therefore $\varphi\circ C_{V^{\otimes d}}=C_M\circ \varphi$,
and that $C_V=n\operatorname{id}_V$. Set $\varphi_C:=\varphi\circ (C_{V^{\otimes d-1}}\otimes \operatorname{id}_V)$. We
get $X^{\mathfrak{S}}_M\varphi=\frac{1}{2}(C_M-n)\varphi-\frac{1}{2}\varphi_C$.

Let $\psi:=\iota(\varphi)$. Clearly, $\iota(C_M\varphi)=C_M\iota(\varphi)$. Also it is easy to check that
$\iota(\varphi_C)=\psi\circ C_{V^{\otimes d-1}}$. But $\psi$ is again $U(\g)$-linear so $\psi\circ C_{V^{\otimes d-1}}=
C_{V^*\otimes M}\circ \varphi$. So $\iota(X^{\mathfrak{S}}_M\varphi)=\frac{1}{2}(C_M-n-C_{V^*\otimes M})\psi$.
It is easy to check that the last expression coincides with $X_{M}^{\GL}\psi$.
\end{proof}


Now we claim that $[\F\mathfrak{S}-\operatorname{mod}]$ is a simple $\hat{\sl}_p$-module. If
$L\in \F\mathfrak{S}_d$-$\operatorname{mod}$ is annihilated by $E$, then $d=0$. Equivalently, there
is only one {\it singular} vertex in the crystal of $\F\mathfrak{S}$-$\operatorname{mod}$, ``singular'' means a
vertex  that is annihilated by
all $\tilde{e}_\alpha$. Since the crystal of an integrable $\hat{\sl}_p$-module is determined uniquely,
we see that an integrable highest weight module whose crystal has only one singular vertex is automatically
irreducible. This proves the claim that the Schur functor categorifies the projection.

In fact, using this categorification we can deduce the classification of irreducible $\F\mathfrak{S}$-modules
together with branching rules, \cite{Kleshchev_br1,Kleshchev_br2}, from the description of the crystal of $\operatorname{Pol}(\GL)$.
We see that the irreducibles are classified by partitions $\lambda=(\lambda_1,\ldots,\lambda_k,\ldots)$ such that
$\lambda_i-\lambda_{i+1}<p$ for all $i$. This is dual to the standard description -- via $p$-restricted
partitions, because the images of the Weyl modules under the Schur functor are dual Specht modules and not
the Specht modules themselves.

\section{What's next?}\label{S_next}
Basically, all ``categories of type A'' occurring in Representation theory carry a categorical action
of a Kac-Moody algebra of type A: $\mathfrak{sl}_m, \mathfrak{gl}_{\infty}$ of $\hat{\mathfrak{sl}}_m$.
These categories include:
\begin{itemize}
\item Categories of representations of (degenerate) cyclotomic Hecke algebras generalizing
$\F\mathfrak{S}$-$\operatorname{mod}$. The categorification functors come from the restriction
and induction and are decomposed using the Jucys-Murphy elements.
\item Categories $\mathcal{O}$ for $\mathfrak{gl}_m$ (or its super, modular, quantum,
affine analogs). The categorification functors come from (suitably understood) tensor
products with the tautological representation and its dual. The decomposition is performed
using the tensor Casimir and its analogs.
\item Categories $\mathcal{O}$ over cyclotomic rational Cherednik algebras, see \cite{Shan}.
\end{itemize}

One can also study categorical actions outside type A. Here the story is more complicated,
categories carrying such actions do not occur in the classical representation theory.
For example, the cyclotomic Hecke algebras should be replaced with Khovanov-Lauda-Rouquier (KLR)
algebras a.k.a. quiver Hecke algebras.

One can also work with categorical actions of quantum groups,
see \cite{Lauda}, \cite{KL1}-\cite{KL2} for details. In this case one works with graded categories.

\end{document}